\documentclass[12pt]{amsart} 
\usepackage{amssymb,extsizes}
\usepackage{amsfonts}
\usepackage{color}

\usepackage{xy}
\xyoption{all}

\usepackage{setspace}
\newtheorem{theorem}{Theorem}[section]
\newtheorem{proposition}[theorem]{Proposition}
\newtheorem{lemma}[theorem]{Lemma}
\newtheorem{corollary}[theorem]{Corollary}
\newtheorem{definition}[theorem]{Definition}

\newtheorem{question}[theorem]{Question}

\theoremstyle{remark}
\newtheorem{remark}[theorem]{Remark}

\newtheorem{example}[theorem]{\bf Example}
\renewenvironment{proof}{{\noindent\bf Proof.}}{\hfill $\Box$\par\vskip3mm}

\newcommand{\End}{{\rm End}}
\newcommand{\Ext}{{\rm Ext}}

\newcommand{\Aut}{{\rm Aut}\,}

\def\CC{{\mathbb C}}

\def\KK{{\mathbb K}}

\begin{document}
\title[Ore Extensions and Infinite Triangularization]{Ore Extensions and Infinite Triangularization}


\begin{abstract}
We give infinite triangularization and strict triangularization results for algebras of operators on infinite dimensional vector spaces. We introduce a class of algebras we call Ore-solvable algebras: these are similar to iterated Ore extensions but need not be free as modules over the intermediate subrings. Ore-solvable algebras include many examples as particular cases, such as group algebras of polycyclic groups or finite solvable groups, enveloping algebras of solvable Lie algebras, quantum planes and quantum matrices. We prove both triangularization and strict triangularization results for this class, and show how they generalize and extend classical simultaneous triangularization results such as the Lie and Engel theorems. We show that these results are, in a sense, the best possible, by showing that any finite dimensional triangularizable algebra must be of this type. We also give connections between strict triangularization and nil and nilpotent algebras, and prove a very general result for algebras defined via a recursive ``Ore" procedure starting from building blocks which are either nil, commutative or finite dimensional algebras.
\end{abstract}

\author{Jeremy R. Edison, Miodrag C. Iovanov, Alex Sistko}
\date{}
\maketitle



\section*{Introduction}

\noindent Let $\KK$ be a field, $n$ a natural number and $V$ an $n$-dimensional $\KK$-vector space. Recall that if $X : V \rightarrow V$ is a $\KK$-linear transformation of an $n$-dimensional $\KK$-vector space $V$, we say that $X$ is \emph{upper-triangularizable} if there is a complete flag $0 = V_0 \le V_1 \le V_2 \le \ldots \le V_{n-1} \le V_n = V$ of $X$-invariant subspaces of $V$. Moreover, we say that $X$ is \emph{strictly upper-triangularizable} if the flag also satisfies $XV_i \subset V_{i-1}$ for all $i \le n$. It is a widely-known fact from linear algebra that if $X$ and $Y$ are two commuting (strictly) upper-triangularizable operators on a finite-dimensional vector space $V$, then they are \emph{simultaneously} (strictly) upper-triangularizable: there is a complete flag of $V$ which is both $X$-invariant and $Y$-invariant. Of course, what is true for the triangularizability of $X$ and $Y$ is also true for each element in the $\KK$-subalgebra of $\End_{\KK}(V)$ which they generate: hence, it also makes sense to talk about \emph{(strictly) upper-triangularizable algebras}. 

Triangularizability results have been investigated from several angles. Two classical theorems from the theory of Lie algebras, Lie's Theorem and Engel's Theorem, yield conditions under which a finite-dimensional solvable (resp. nilpotent) Lie algebra is upper-triangularizable (resp. strictly upper triangularizable) \cite{H}. Beyond this, several authors have studied infinite-dimensional versions of upper-triangularizable algebras, often through analytic and operator-theoretic means \cite{M1,M2}. In this treatment, a subalgebra $A \subset \End_{\KK}(V)$ of operators on a (possibly infinite-dimensional) vector space $V$ is upper triangularizable if $V$ contains a well-ordered basis with respect to which elements of $A$ act in a triangular fashion. Essentially, this means that $V$ has a (possibly infinite) composition series as an $A$-module with one-dimensional simple factors. In module theoretic terms, this means that $V$ is a semiartinian $A$-module with 1-dimensional simple factors. A particularly important case of this is when such a module is also locally finite; when the algebra is Noetherian, triangularization reduces to this special type of locally finite modules (and this is the case in many classical situations of simultaneous triangularization). We will say that such a locally finite module with 1-dimensional simple factors is \emph{locally p-finite} (where p stands for ``pointed"). The ``strict'' version of this definition requires that the simple factors all lie in a single isomorphism class, in which case we call the module \emph{locally p-finite homogeneous}.  

Often, triangularization results are statements which say that operators satisfying certain types of algebraic equations are simultaneously triangularizable. Hence, instead of starting with a vector space $V$ and asking which subalgebras of $\End_{\KK}(V)$ are triangularizable, one could start instead with an associative algebra $A$ (given by certain type of relations) and ask which of its modules are locally p-finite or locally p-finite homogeneous. This allows for a purely algebraic and representation-theoretic treatment of (strict) upper-triangularizability. In this article, we provide such a treatment for not-necessarily finite-dimensional associative algebras. To do this, we introduce analogues to solvable/nilpotent groups and Lie algebras in the context of associative algebras and provide conditions under which they yield triangularizable modules. We call these algebras \emph{Ore-solvable algebras} (due to their similarity to Ore extensions). Specifically, we will say that an algebra $A$ is Ore-solvable if there exists of a sequence of subalgebras $\KK = A_0 \subset A_1 \subset \ldots \subset A_{n-1} \subset A_n = A$ and elements $x_i \in A_i$ (for $0 \le i \le n$) satisfying the following condition:

\vspace{.2cm}

($*$) for each $i\geq 1$, $A_i$ is generated by $A_{i-1}$ and $x_i$ and $A_{i-1}x_i + A_{i-1} = x_iA_{i-1} + A_{i-1}$.

\vspace{.2cm}
 

Ore-solvable algebras contain many well-known classes of algebras, including: finitely generated commutative algebras; enveloping algebras of solvable or nilpotent Lie algebras; group algebras of finite solvable groups or of polycyclic groups; and iterated Ore extensions. In particular, this includes various types of skew polynomial rings, quantum matrix algebras and quantum affine algebras. Such algebras are of importance in many fields and have received much attention. 

The above condition ($*$) implies the existence of functions $\sigma_i, \theta_i : A_{i-1} \rightarrow A_{i-1}$ for each $i$ satisfying $ax_i = x_i\sigma_i(a) + \theta_i(a)$. We refer to the collection $(x_i,\sigma_i,\theta_i)$ as \emph{left Ore datum} for $A$. Similarly, we will also have a right version of the Ore datum $(x_i,\sigma_i',\theta_i')$; hence, we will have left- and right- versions of Ore-solvable algebras which satisfy one-sided versions of condition ($*$) above. A single algebra $A$ may have many different collections of Ore datum. The triangularizability results of this article are often phrased in terms of conditions on the Ore datum of an Ore-solvable algebra. 

 Our main results on upper-triangularizability are Theorems \ref{t.genlie}, \ref{t.genlie2} and \ref{t.3} which can be combined into the following: 

\begin{theorem} 
Let $A$ be an Ore-solvable algebra over a field $\KK$ with Ore-datum $(x_i,\sigma_i,\sigma_i',\theta_i,\theta_i')$. Suppose that for each $i$, we can find a generating set $S_i$ for $A_{i-1}$ such that at least one of the following holds: 
\begin{enumerate} 
\item ${\rm char}(\KK)=0$, $\theta_i(a) \in [A_i,A_i]$ for each $a \in S_i$, and the semigroup generated by $\sigma_i$ acts trivially on the one-dimensional algebra characters of $A_{i-1}$ (considered inside the set of functions $A\rightarrow A$); or
\item ${\rm char}(\KK)=0$, $\theta_i(a) \in [A_i,A_i]$ for each $a \in S_i$, and $\sigma_i$ is an algebra endomorphism of $A_{i-1}$ such that whenever $\lambda$ is an algebra character of $A_{i-1}$ and $k$ a positive integer with $\lambda \neq \lambda\circ \sigma_i^k$, then $\Ext^1(\lambda\circ \sigma_i^k,\lambda) = 0$; or 
\item For all $a\in S_i$, either $\theta_i(a)$ is nilpotent or $\theta_i(a) \in [A_{i-1},A_{i-1}]$, and the orbits of algebra characters of $A_{i-1}$ under the action of $\sigma_i$ are each either trivial or infinite; 
\end{enumerate}
\noindent Then every simple $A$-module is $1$-dimensional, and a $A$-module $V$ is $p$-finite (equivalently, triangularizable in this case) if and only if each $x_i$ acts locally finitely on $V$. In particular, such an algebra with a $p$-finite faithful module is upper-triangularizable.
\end{theorem} 

\noindent This result can be seen as a generalization of Lie's classical theorem on simple representations of finite-dimensional solvable Lie algebras. In Theorem \ref{t.na1} we also prove the following result on strict upper- triangularizability:  

\begin{theorem}
Let $A$ be a left Ore-solvable algebra, $V$ an $A$-module and assume that the Ore datum of $A$ satisfies the following two conditions: 
\begin{enumerate}
\item Each algebra $A_i$ from the filtration of $A$ is generated by a set $S_i$, such that for each $a\in S_{i-1}$, $\theta_i(a)$ acts locally nilpotent on $V$; and
\item each $x_i$ is a locally nilpotent operator on $V$.  
\end{enumerate}
Then $V$ is locally p-finite homogeneous (equivalently, strictly triangularizable), and every simple subquotient of $V$ has character $\mu:A\rightarrow \KK$ which satisfies $\mu(x_i)=0$ for all $i$.
\end{theorem}

We also note a converse of the above, which shows that if $V$ is finite dimensional, then every triangularization or strict triangularization must come from an Ore situation as above, which is of special type and close to a Lie algebra. We show every finite dimensional (strictly) triangularizable algebra can be presented as an Ore algebra with subalgebras forming a complete flag, whose Ore datum is of the type $\sigma_i={\rm Id}$ and $\theta_i(a)=[x_i,a]$. We note that strictly upper-triangular finite-dimensional algebras are nothing more than local connected algebras, and show that they form a class of Ore-solvable algebras satisfying natural nilpotency conditions. More precisely, Propositions \ref{p.OreNilAlg} and \ref{FinDimLCA} combine to yield the following:

\begin{theorem} 
Let $A$ be a $k$-algebra. Then the following are equivalent: 
\begin{enumerate} 
\item $A$ is a finite-dimensional strictly-upper triangularizable algebra.  
\item $A$ is a finite-dimensional local connected algebra.
\item $A$ is a left Ore-solvable algebra with Ore datum $(x_i,\sigma_i, \theta_i)$ satisfying the following: 
\begin{enumerate} 
\item $\sigma_i={\rm Id}$, $\theta_i(-)=[x_i,-]$.
\item There is a set $S_i$ of  generators of $A_i$ such that for each $a \in S_{i-1}$, $\theta_i(a)$ is nilpotent.
\item The $x_i$ are nilpotent for each $i$.
\end{enumerate}
\end{enumerate}
\end{theorem}

In the last section, we examine the connection between nil algebras, nilpotent algebras and strict triangularizability. This is motivated by classical results, such as the fact that finitely many commuting operators on a finite vector space are simultaneously triangularizable, or Levitski's theorem on finite semigroups of nilpotent matrices. We show that the strict triangularizability of finitely many operators on an arbitrary vector space reduces to the locally (p)-finite triangularizability mentioned before; that is, we show that a strictly triangularizable module over a finitely generated algebra must be locally p-finite. We say that an algebra is connected if, up to isomorphism, it has a unique type of finite dimensional simple module which is furthermore of dimension 1. We also introduce the idea of recursively building algebras from pieces, similar to procedures used in solvable groups or Lie algebras; briefly, if P is a property of an non-unital algebra, we say $A$ is $n$-recursively P if there exists a sequence of subalgebras $0=A_0\subseteq A_1\subseteq \dots \subseteq A_n=A$ and subalgebras $B_i\subseteq A_i$ such that $B_iA_{i-1}\subseteq A_{i-1}B_{i}+A_{i-1}$ and $A_i$ is generated by $B_i$ and $A_{i-1}$, and the $B_i$ are $k_i$-recursively P for some $k_i<n$. We say that a $0$-recursively P algebra is simply one with property P, and a recursively P algebra is one which is $n$-recursively P for some $n$. We call the initial algebras $B_i$ used at step $0$ the constituents of $A$. For unital algebras the same definition applies, but we need to include non-unital algebras in order to have nil and nilpotent algebras.  The following is our general strict simultaneous triangularization result, in the spirit of other results stating that if certain operators are triangularizable, then they are simultaneously triangularizable.


\begin{theorem} 
Let $A$ be a recursively cyclic-finite algebra (an algebra built recursively as above from either 1-generated algebras or finite dimensional algebras). Then an $A$-module $V$ is strictly triangularizable if and only if each of the constituents of $A$ act locally nilpotent on $V$. In this case, every $A$-module is locally p-finite homogeneous. 
\end{theorem}   

We examine several of the recursive classes introduced, and also note, as a byproduct, the fact that a recursively finitely-generated-nil algebra is nil.


This paper is organized as follows. With a broader audience in mind, in Section 1 we introduce necessary background on semiartinian modules and triangularizability. In Section 2 we introduce Ore-solvable algebras, and prove our main results on triangularizability. In Section 3 we turn our attention to strict upper-triangularizability for Ore-solvable algebras, and in Section 4 we give the above mentioned connections and strict triangularization results for algebras generalizing nil algebras. 


{

\section{Preliminaries}\label{s.1}

\subsection*{Semiartinian modules}

Given an algebra $A$ over a field, recall that an $A$-module $M$ is \textit{semiartinian} (or a \textit{Loewy module}) if every non-zero quotient of $M$ has nonzero simple submodule. Equivalently, there exists a well ordered sequence $(M_{i})_{i\in I}$ of submodules of $M$, such that $M_i \subset M_j$ if $i \leq j$, $M_{i+1}/M_i$ is semisimple for each $i \in I$, $M_{\alpha} = \bigcup\limits_{j < \alpha} M_j$ for each limit ordinal $\alpha \in I$ and
$M = \bigcup\limits_{i \in I} M_i$. This is equivalent to the existence of such a sequence where each $M_{i+1}/M_i$ is simple, and we will call such a sequence a {\it composition series} for $M$.  
Recall that the Loewy series of a module $M$ is defined recursively as follows: $L_0(M)=soc(M)$ (the socle of $M$) and for each ordinal $\alpha$, 

\begin{enumerate}
\item $M_{\alpha}$ is such that $M_{\alpha}/M_{\alpha-1}=soc(M/M_{\alpha-1})$, if $\alpha=\beta+1$ is a successor, and

\item $M_\alpha=\bigcup\limits_{\beta<\alpha}M_\beta$ if $\alpha$ is a limit ordinal.  
\end{enumerate}  


\noindent Then $M$ is semiartinian if and only if the Loewy series of $M$ terminates at $M$, so $M=M_\alpha$ for some ordinal $\alpha$; the smallest such ordinal is called the Loewy length of $M$. It is also known that if $M$ is semiartinian, and $(M_\alpha)_\alpha$ and $(N_\beta)_\beta$ are two composition series, then for each simple $A$-module $S$, the cardinalities of the sets $\{\alpha| M_{\alpha+1}/M_\alpha\}\cong S\}$ and $\{\beta| N_{\beta+1}/N_\beta\}\cong S\}$ are the same. The proof is similar to that in the finite case. More specifically, the Schreier Refinement Theorem (which uses only the Zassenhaus Butterfly Lemma) can be extended to sequences of submodules indexed by well-ordered sets, which allows one to conclude isomorphism of composition series. Details can be found in \cite{Sh1, Sh2} (see also \cite{R}). 

We introduce here the following variations of this notion, which we will see are special types of semiartinian modules closely related to triangularizability. We recall that an $A$-module is locally finite (sometimes called locally finite dimensional) if it is the sum of its finite dimensional submodules.

\begin{definition}
Let $M$ be a semiartinian $A$-module.  
\begin{enumerate}
\item We say that $M$ is f-semiartinian if every simple subquotient of $M$ is finite dimensional. 
\item We say that $M$ is p-semiartinian, or pointed semiartinian, if every simple subquotient of $M$ is 1-dimensional.
\item We say that $M$ is locally p-finite if it is locally finite and every simple subquotient of $M$ is 1-dimensional, equivalently, it is locally finite and p-semiartinian. 
\item We say that $M$ is homogeneous if $M$ has a unique type of simple subquotient $S$ up to isomorphism; that is, there is a simple module $S$ such that every simple subquotient of $M$ is isomorphic to $S$. In this case we will also say that $M$ is $S$-homogeneous.  
\end{enumerate}
\end{definition}

Using the previous remarks, we immediately notice that $M$ is f-semiartinian if and only if it has a composition series with finite dimensional (simple) composition factors, and it is p-semiartinian if and only if it has a composition series with 1-dimensional factors. 
 
Let $\KK$ be a field, $V$ be a vector space, $X\subset \End(V)$ a subset, and let $A$ be the subalgebra of $\End(V)$ generated by $X$. Then $X$ is called \textit{diagonalizable} if there exists a basis of eigenvectors for $X$, equivalently, there exists a basis of eigenvectors for $A$. We say that $X$ is \textit{block diagonalizable}, if there exists a basis $B=\bigcup\limits_iB_i$ partitioned into finite parts $B_i$ such that each $B_i$ spans an $X$-invariant subspace. Again, note that $X$ is block diagonalizable if and only if $A$ is block diagonalizable. In representation theoretic terms, $A=\KK\{X\}\subset \End(V)$ is diagonalizable (resp. block diagonalizable) if and only if $V$ is a direct sum of 1-dimensional (resp. finite dimensional) modules. Recall from \cite{M1} that $X$ is \textit{triangularizable} if there exists a well ordered basis $(B,\leq)$ such that for all $v \in B$ and all $a \in X$, $a(v) \subset \operatorname{Span}_{\KK}\{u \in B : u \leq v \}$. It is easy to note that $X$ is triangularizable if and only if $A$ is triangularizable. $X$ is \textit{strictly triangularizable} if there is such a well ordered basis such that $a(v) \subset \operatorname{Span}_{\KK}\{u : \in B : u < v \}$ for all $a\in X$. 
 We will say that $X$ is \textit{block triangularizable} if there exists a collection of finite sets $B_i\subseteq V$ indexed by a well-ordered set $I$, such that the following hold:  
 \begin{enumerate}
\item The sets $B_i$ are mutually disjoint.
\item The set $B=\bigsqcup\limits_iB_i$ is a basis of $V$.
\item For each $i\in I$ and each $a\in X$, we have $a\cdot B_i\subseteq \sum\limits_{j\leq i}{\rm Span}(B_j)$. 
\end{enumerate}
\noindent Again, $X$ is block triangularizable if and only if $A$ is so. When studying triangularizability of a set $X$ we may thus work with the algebra generated by $X$, and it will be useful to work with the module/representation theoretic  re-interpretation of these notions, which we briefly state here.

\begin{proposition}
Let $A\subseteq \End(V)$ be a subalgebra, as above. Then $A$ is triangularizable if and only if $V$ is a p-semiartinian module, and $A$ is block triangularizable if and only if $V$ is f-semiartinian.
\end{proposition}

\begin{proof}
\cite[Proposition 8]{M2} 
states that an algebra generated by a set $X \subseteq \End(V)$ is triangularizable if and only if there is a well-ordered collection $(V_i)_i\in I$ of $X$-invariant subspaces of $V$, which is maximal as a well-ordered set of subspaces of $V$. This means $V_{i+1}/V_i$ is 1-dimensional for each $i\in I$, and this is equivalent to $V$ being p-semiartinian as module over the algebra $A = \KK\{X\}$ generated by $X$. The second assertion is proved similarly, and may obtained by applying the same argument as the aforementioned proposition from \cite{M2}, \textit{mutatis mutandis}.
\end{proof}

 We also note the module theoretic interpretation of strict triangularization. 

\begin{proposition}\label{p.strictTR}
Let $X\subseteq \End(V)$ and let $A$ be the subalgebra of $\End(V)$ generated by $X$. Then $X$ is strictly triangularizable if and only if $V$ is a homogeneous p-semiartinian $A$-module, such that its unique (up to isomorphism) simple submodule $S\cong \KK w$ has the property that $a\cdot w=0$ for all $a\in X$. 
\end{proposition}
\begin{proof}
$X$ is strictly triangularizable if and only if there is a well-ordered collection $(V_i)_i\in I$ of $X$-invariant subspaces of $V$, which is maximal as a well-ordered set of subspaces of $V$, and such that $X\cdot V_{i+1}\subset V_i$. This is equivalent to asking that $(V_i)_i$ is a composition series (in the infinite sense defined above) and each simple $V_{i+1}/V_i$ is spanned by a vector $\overline{v_{i+1}}$ such that $a\cdot \overline{v_{i+1}}=0$ for all $a\in X$. This means all these simple modules are isomorphic as $A=\KK\{X\}$-modules ($A$ is the subalgebra of $\End(V)$ generated by $X$) and are described precisely by the property in the statement of the proposition.
\end{proof}

\subsection*{Triangularizable modules}

By a slight extension of the previous terminology, we may consider an arbitrary algebra $A$ and $V$ an $A$-module, and consider the algebra $B=A/{\rm ann}_A(V)$ as a subalgebra of $\End(V)$. We will then say that $V$ is a triangularizable/block triangularizable $A$-module if $B$ is triangularizable; this is equivalent to saying that $V$ is p-semiartinian/f-semiartinian. The alternate ``triangularizable" terminology simply emphasizes the connection to this notion. It will be of use as we will also take the alternate viewpoint of starting with a certain algebra possibly given by generators and relations, and determining what modules have such triangularization properties.


\section{Ore-Solvable Algebras and Triangularizability} 

\subsection{Ore-Solvable Algebras}

We first make the following observation, which will motivate the formulation of our results. Recall that a module $V$ is \textit{locally of finite length} if it is the sum of its finite length submodules. Note that every module which is locally of finite length is semiartinian (since each module of finite length is semiartinian and the class of semiartinian modules is closed under subquotients and direct sums).

\begin{proposition}
Let $R$ be a Noetherian subalgebra of $\End(V)$ for a vector space $V$. Then $V$ is a semiartinian $R$-module if and only if $V$ is locally of finite length. In particular, $R$ is block triangularizable if and only if $V$ locally finite, and $R$ is triangularizable if and only if $V$ is locally p-finite.
\end{proposition}

\begin{proof}
We only need to prove the ``only if'' part. Suppose $V$ is a semiartinian $R$-module, and let $x \in V$. Then $M = Rx$ is a finitely generated, and thus Noetherian, $R$-module. Since $M$ is Noetherian and semiartinian, it is in fact Artinian of finite length: indeed, its Loewy series is finite $L_0\subset \dots \subset L_n=A$ and each semisimple $L_i/L_{i-1}$ has finite length, since it is also Noetherian. The conclusion now follows. 
\end{proof}

The above shows that to investigate triangularizability or block-triangularizability for (modules over) algebras that are Noetherian, one can equivalently study locally finite modules. Thus, the situation where $V$ is not just semiartinian but locally finite can be regarded as a special type of triangularization which is of particular interest. For example, the case of finitely-many commuting operators $T_i$ on a vector space $V$ falls under this set-up of Noetherian subalgebras of $\End(V)$. That is, a finite family of commuting operators $\{T_i\}_{i=1,\dots,n} \subset \End(V)$ is simultaneously triangularizable if and only if $V$ is a locally p-finite module. As we will note, many other situations will also be of this type.   

We first introduce a definition. This definition can be carried over without change for algebras over a commutative ring, but for our purposes we restrict attention to algebras over a field.

\begin{definition}\label{d.solvable}
Let $A$ be a $\KK$-algebra. We say that $A$ is solvable if there exists a sequence $\KK=A_0\subset A_1\subset \dots \subset A_n=A$ of subalgebras and commutative subalgebras $B_i\subset A$ such that $A_{i-1}B_i+A_{i-1}=B_iA_{i-1}+A_{i-1}=A_{i}$ for all $i$.
\end{definition}


Our first result reduces block-triangularization of a solvable algebra to checking it for its building pieces $B_i$.

\begin{theorem}\label{t.1st}
Let $A$ be a not necessarily unital algebra, and assume that there exists a sequence $\KK=A_0\subset A_1\subset \dots \subset A_n=A$ of subalgebras, as well as subalgebras $B_i\subset A$ such that the following conditions hold: 
\begin{enumerate}
 \item $A_{i-1}B_i \subseteq B_iA_{i-1}+A_{i-1}$, and  
 \item $A_i$ is generated by $A_{i-1}$ and $B_i$ for all $i$. 
 \end{enumerate}
\noindent Then a left or right $A$-module $V$ is locally finite if and only if $V$ is a locally finite $B_i$-module for all $i$. 
\end{theorem}

\begin{proof}
Note first that the hypotheses imply that $B_iA_{i-1}+A_{i-1}=A_{i}$ for all $i$. It is clear that if a module is locally finite as an $A$-module, then it is locally finite as a $B_i$-module for all $i$. Conversely, we proceed by induction on the length $n$ of such a chain of subalgebras. The case $n=1$ is trivial. For the induction step, let $x\in V$. If $V$ is a left module, then $Ax=A_nx=B_nA_{n-1}x+A_{n-1}x$; but $A_{n-1}x$ is finite dimensional, and therefore $B_nA_{n-1}x$ is finite dimensional by hypothesis, and thus $\dim(Ax)<\infty$. Similarly, if $V$ is a right $A$-module, $xA=xB_nA_{n-1}+xA_{n-1}$; then $xB_n$ is finite dimensional, and the induction hypothesis implies $xB_nA_{n-1}$ and $xA_{n-1}$ are both finite dimensional.
\end{proof}

We will often be interested in the situation where the commutative subalgebras $B_i$ are finitely generated, or even generated by a single element. In this situation the previous theorem may be rephrased as below.

\begin{corollary}\label{c.locfincond} Let $A$ be an algebra. Then the following statements hold: 
\begin{enumerate}
\item[(i)] Suppose that $A$ is commutative an has a finite generating system $a_1,\dots,a_n$. Then a module $V$ is locally finite if and only if each $a_i$ acts locally finite on $V$. 
\item[(ii)] Suppose that $A$ is a solvable algebra, such that each of the algebras $B_i$ are finitely generated by some elements $x_{ij}$ (thus $A$ is finitely generated by the finite collection $x_{ij}$). 
Then an $A$-module $V$ is locally finite if and only if each $x_{ij}$ acts locally finite on $V$.  
\end{enumerate}
\end{corollary}
\begin{proof}
For the first statement, one may see the commutative algebra $A$ as a solvable algebra with $B_i=\KK[x_i]$ and $A_i=\KK[x_1,\dots,x_i]$, and apply the previous Theorem. (Note that any commutative algebra is always trivially solvable); the second statement follows from the first and the previous Theorem.
\end{proof}

\noindent For our next results, we introduce another definition. 

\begin{definition}
Let $A$ be a (possibly non-unital) $\KK$-algebra. We say that $R$ is left (respectively, right) Ore-solvable if there is a sequence of subalgebras $\KK=A_0\subset A_1\subset \dots\subset A_n=A$ and cyclic subalgebras $B_i=\KK[x_i]$ 
for some $x_i\in A_i$, such that $A_i$ is generated by $A_{i-1}$ and $x_i$ and $A_{i-1}x_i\subseteq x_iA_{i-1}+A_{i-1}$ (respectively, $x_iA_{i-1} \subseteq A_{i-1}x_i+A_{i-1}$) for each $i$. We say that the extension is Ore-solvable if there exist such a sequence of subalgebras such that both of these containments hold.

\end{definition}

\noindent That is, an Ore-solvable algebra is a solvable algebra whose successive algebras $B_i$ are generated by just one element, with a ``global commuting" condition on $x_i$ and $A_i$. It follows immediately from Corollary \ref{c.locfincond} that if $A$ is Ore-solvable and $V$ is an $A$-module, then $V$ is locally finite if and only if each $x_i$ acts locally finite on $V$.


We note that the above conditions in the definition of an Ore-solvable algebra imply and require the existence of maps $\sigma_i,\theta_i:A_{i-1}\rightarrow A_{i-1}$ such that $ax=x\sigma_i(a)+\theta_i(a)$ for all $a\in A_{i-1}$, as well as $\sigma_i',\theta_i':A_{i-1}\rightarrow A_{i-1}$ with $xa=\sigma_i'(a)x+\theta_i'(a)$. In the classical case of Ore extensions, of course, $\sigma_i,\sigma_i'$ are automorphisms of the algebra $A_{i-1}$ (sometimes just algebra endomorphisms), the $\theta_i,\theta_i'$ are ($\sigma_i$-${\rm id}$)-skew derivations, and $A_i=\bigoplus\limits_nA_{i-1}x_i^n$ is a free module. We do not require any such module property of $A_i$ over $A_{i-1}$. We will also say that $A$ is an {\it Ore-solvable algebra with Ore-datum $(x_i,\sigma_i,\sigma_i',\theta_i,\theta_i')$}. Evidently this Ore-datum may not be uniquely determined, and the same algebra structure can be endowed with structure of an Ore-solvable algebra by different sets of Ore data. Throughout, whenever we refer to an Ore-solvable algebra, we will implicitly assume and use the above notations.

\begin{example}\label{e.solvgp}
Let $G$ be a polycyclic group and consider a series $1 = G_0 \triangleleft G_1 \triangleleft \ldots \triangleleft G_n = G$ such that $G_i/G_{i-1}$ is cyclic for each $i$, say generated by (the image of) some $x_i \in G_i$. Let $A = \KK G$ be the group algebra, and for each $i$ set $A_i = \KK G_i$ and $B_i = \KK[x_i]$.  Then clearly we have a filtration $\KK = A_0 \subset A_1 \subset \ldots \subset A_n = A$. It is easy to see that $A$ is Ore-solvable, since if $g\in G_{i-1}$, then $x_ig=(x_i g x_i^{-1}) x_i+0$ and the Ore datum is given by $\sigma_i'(g)=x_i g x_i^{-1}\in G_{i-1}$ for any $g \in G_{i-1}$ ($\theta_i$ is simply the zero map). In particular, this applies for any finite solvable group. 
\end{example}

\begin{example}\label{e.solvlie} 
Let $\mathbb{K}$ be a field of characteristic 0 and $\mathfrak{g}$ be a solvable Lie algebra of dimension $n$ with elementary sequence $\mathfrak{g} = \mathfrak{g}_n \supset \mathfrak{g}_{n-1} \supset \ldots \supset \mathfrak{g}_1 \supset \mathfrak{g}_0 = 0$. Then each $\mathfrak{g}_i/\mathfrak{g}_{i-1}$ is one-dimensional, say with $x_i \in \mathfrak{g}_i$ inducing a basis for $\mathfrak{g}_i/\mathfrak{g}_{i-1}$.  For each $i$, let $A_i = U(\mathfrak{g}_i)$ denote the universal enveloping algebra of $\mathfrak{g}_i$ and  $B_i = \mathbb{K}[x_i]$. Then 
for any $y \in \mathfrak{g}_{i-1}$, the relation $x_i y = yx_i + [x_i,y]$ holds (and note that $[x_i,y] \in \mathfrak{g}_{i-1}$).  This shows that $A$ is Ore-solvable with Ore-datum  $\sigma_i'(y) = y$ and $\theta_i'(y) = [x_i,y]$.
\end{example}

The following will be relevant for our results. It is very similar to the well-known fact that iterated Ore extensions are Noetherian, and it has weaker versions for left/right Ore solvable algebras. A proof follows directly from Theorem 2.8.1 of \cite{B}. 
\begin{theorem}
Let $A$ be an Ore solvable algebra. Then $A$ is Noetherian.
\end{theorem}
\noindent We also note the following straightforward fact: 
\begin{proposition}\label{p.fdsimplpf}
Let $A$ be an algebra. Then the following are equivalent: 
\begin{enumerate}
\item Every finite dimensional simple $A$-module is 1-dimensional.
\item Every locally finite $A$-module is locally p-finite. 
\end{enumerate}
\end{proposition}

The following Lemma contains the main computation used in all the results that follow. \\

\begin{lemma}\label{l.comp}
Let $A\subseteq B$ be an algebra extension, and $V$ a $B$-module. Assume that:
\begin{enumerate}
\item[(i)] $B=A\cdot \KK[x]$ and  $Ax\subseteq xA+A$, so there are $\sigma,\theta:A\rightarrow A$ with $ax=x\sigma(a)+\theta(a)$, for all $a\in A$. 
\item[(ii)] $V$ is finite dimensional and contains an $A$-eigenvector $v\in V$ with weight $\lambda:A\rightarrow \KK$ (so $av=\lambda(a)v$ for $a\in A$), such that $V$ is generated by $v$ as a $\KK[x]$-module.  
\end{enumerate}
Then:
\begin{enumerate}
\item[(a)] The action of $a\in A$ on $V$ is given by:
\begin{eqnarray*}
av & = & \lambda(a)v;\\
a\cdot xv & = & \lambda(\sigma(a))xv+\lambda(\theta(a))v;\\
(*)\,\,\,\,\,\,\,\,\,\,\,\,\,a\cdot x^2v & = 
& \lambda(\sigma^2(a))x^2 v + \lambda((\sigma\circ\theta)(a)+(\theta\circ\sigma)(a))xv+\lambda(\theta^2(a))v\\
& \dots & \\
a\cdot x^iv & = &  \lambda(\sigma^i(a))x^iv + \sum\limits_{j<i}\mu_{ij}(a)x^{j}v\\
& \dots & 
\end{eqnarray*}
where for $i<j$, $\mu_{ij}:A\rightarrow \KK$ are linear maps which are sums of compositions of $\sigma$ and $\theta$.
\item[(b)] If $A$ is generated by a set $S$ such that $\theta(a)$ acts nilpotently on $V$ for each $a\in S$, then $V$ is a semisimple p-finite $A$-module with weights $\lambda\circ \sigma^k$; that is $V\cong \bigoplus\limits_{i=0}^{n-1}\KK x^iv$ as $A$-modules, and each $\KK x^iv$ is a $A$-module of weight $\lambda\circ \sigma^i$. 
\end{enumerate}
\end{lemma}
\begin{proof}
(a) Note that $a\cdot xv  = x\sigma(a)\cdot v+\theta(a)v= \lambda(\sigma(a))xv+\lambda(\theta(a))v$, and $a\cdot x^2v = (ax)(xv) = (x\sigma(a)+\theta(a))(\lambda(\sigma(a))xv+\lambda(\theta(a))v) = \lambda(\sigma^2(a))x^2 v + \lambda((\sigma\circ\theta)(a)+(\theta\circ\sigma)(a))xv+\lambda(\theta^2(a))v$. The general formula follows by an easy induction. \\
(b) By (a), the action of $a$ on $V$ is upper triangular in a basis $\{v,\dots,x^kv\}$, with eigenvalues $[\lambda\circ \sigma^i](a)$. If $\theta(a)$ acts nilpotently for $a\in S$, then $\lambda(\theta(a))=0$ and so $a\cdot xv=\lambda(\sigma(a))\cdot xv$. Since this holds for $a\in S$, we immediately observe that $f(a_1,\dots,a_n)\cdot xv\in \KK xv$ for any (noncommutative) polynomial $f$ and $a_1,\dots,a_n\in A$, and so $\KK xv$ is $A$-invariant. Hence, we must have $a\cdot xv=[\lambda\circ \sigma](a)\cdot xv$ for all $a\in A$, and thus $\lambda\circ\sigma$ is a weight. Applying the same argument for the $A$-eigenvector $xv$ we obtain again that $x^2v$ is an $A$-eigenvector, of weight $\lambda\circ \sigma^2$. Hence, inductively we get $\KK x^iv$ are all simple 1-dimensional $A$-modules, and they span $V$, from which the conclusion follows. 
\end{proof}

\noindent We apply the above computational Lemma to several types of triangularization results.

\subsection{Characteristic Zero}

\begin{theorem}\label{t.genlie}
Suppose $\KK$ is an algebraically closed field of characteristic $0$. Let $A$ be an Ore-solvable algebra such that the Ore datum satisfies: 
\begin{enumerate}
\item[(i)] For each $i$, $A_{i-1}$ is generated by a set $S_i$ such that $\theta_i(a)\in [A_{i},A_{i}]$ for all $a\in S_i$;
\item[(ii)] For each $i$, the semigroup generated by $\sigma_i$ acts trivially on the set of 1-dimensional characters of $A_{i-1}$ ($\lambda\circ\sigma_i=\lambda$ for every 1-dimensional character $\lambda$ of $A_{i-1}$). 
\end{enumerate}
Then the following hold: 
\begin{enumerate}
\item Every finite dimensional simple $A$-module is 1-dimensional.  
\item If $V$ is an $A$-module, then $V$ is a locally p-finite $A$-module (equivalently, locally finite, or triangularizable in this case) if and only if each $x_i$ acts locally finite on $V$.  
\end{enumerate} 
\end{theorem}
\begin{proof}
(1) We argue by induction on $n$: for $n=0$ it is obvious, so let $n\geq 1$. For simplicity of notation let $x = x_n$, $\sigma = \sigma_n$ and $\theta = \theta_n$.  Let $V$ be a finite dimensional simple $A$-module. Since $V$ is finite dimensional, it contains a simple $A_{n-1}$-submodule. This submodule must be 1-dimensional by the inductive hypothesis, and therefore it is generated by an $A_{n-1}$-eigenvector $v \in V$ and has an associated weight $\lambda:A_{n-1}\rightarrow \KK$. By Lemma \ref{l.comp} (a), if $a\in S_n$ then $\theta(a)$ acts with eigenvalues $[\lambda\circ\sigma^i](\theta(a))$; by (b), these are all equal to $\lambda(\theta(a))$, and so $Tr(\theta(a))=n\lambda(a)=0$ since $\theta(a)\in [A_n,A_n]$ (as an endomorphism of $V$, $\theta(a)$ is a commutator). By the characteristic assumption we get $\lambda(a)=0$, and hence $\theta(a)$ is nilpotent. Now by Lemma \ref{l.comp} (b), $V$ is a direct sum of 1-dimensional modules, which are all isomorphic (since $\lambda\circ\sigma^i=\lambda$), and the action of $A$ on $V$ is simply given by $a\cdot w=\lambda(a)w$ for each $w\in V$. Hence, choosing $w$ to be an $x$-eigenvector, it will then also be a $B$-eigenvector, so $V=\KK w$ by the simplicity of $V$.\\
(2) By Proposition \ref{p.fdsimplpf}, every locally finite $A$-module is in fact locally p-finite. The result follows from this observation together with Corollary \ref{c.locfincond}.
\end{proof}

\noindent We also note the following second triangularization result for Ore-solvable algebras, where the condition on $\sigma$ is a homological one. We will use the following standard fact of finite dimensional algebras. 

\begin{lemma}
Let $V$ be a finite dimensional indecomposable module over an algebra $H$, with a composition series $0=V_0\subseteq V_1\subset \dots\subset V_n$ with (simple) quotients $S_i=V_i/V_{i-1}$. Then there is $i<j$ such that $\Ext^1_B(S_j,S_i)\neq 0$. 
\end{lemma}

\begin{theorem}\label{t.genlie2}
Suppose $\KK$ is an algebraically closed field of characteristic $0$. Let $A$ be an Ore-solvable algebra, satisfying the following conditions: 
\begin{enumerate}
\item[(i)] For each $i$, $A_{i-1}$ is generated by a set $S_i$ such that $\theta_i(a)\in [A_{i},A_{i}]$ for all $a\in S_i$;\\
\item[(ii)] For each $i$, $\sigma_i$ is an algebra endomorphism of $A_{i-1}$ such that if $\lambda$ is a 1-dimensional character of $A_{i-1}$, then $\Ext^1(\lambda\circ\sigma_i^k,\lambda)=0$ (as $A_{i-1}$-modules), whenever $\lambda\neq \lambda\circ\sigma^i$.  
\end{enumerate}
Then the conclusions of Theorem \ref{t.genlie} hold: every finite dimensional module is 1-dimensional, and an $A$-module $V$ is p-finite (equivalently, locally finite) if and only if every $x_i$ acts locally finite on $V$. 
\end{theorem}
\begin{proof}
We proceed by induction as in the proof of Theorem \ref{t.genlie}, and use the same notation as there for a finite dimensional simple $A_n$-module $V$. Using the calculation of Lemma \ref{l.comp}, we interpret the matrix of $a\in A_{n-1}$ in the basis $v,xv,\dots,x^{n-1}v$ to show that $V$ has a composition series with 1-dimensional quotients with characters $\lambda,\lambda\circ\sigma,\dots,\lambda\circ\sigma^{n-1}$, in this order. For each $i<j$, hypothesis (ii) applied to the character $\lambda\circ\sigma^{j-i}$ shows that $\Ext^1(\lambda\circ\sigma^{j},\lambda\circ\sigma^i)=0$ whenever $\lambda\circ\sigma^i\neq \lambda\circ\sigma^j$. Using block theory and the previous Lemma, since there are no extensions between non-isomorphic simple $A_{i-1}$-modules among those having characters $\lambda\circ \sigma_i$, we conclude that $V$ splits as an $A_{i-1}$-module into a direct sum of modules $V=\bigoplus_{\mu\in F} V_\mu$, each of which is homogeneous in the sense that the composition factors of $V_\mu$ are all isomorphic to the same (1-dimensional character) $\mu$. Moreover, for each such $\mu$ we have $\mu\circ\sigma=\mu$. Using again the computation of Lemma \ref{l.comp}, and starting with some $A$-eigenvector $w$ of weight $\mu$, we obtain a submodule of $V_\mu$, spanned by some $w,\ldots,x^kw$ and invariant under the action of $x$. Since $V$ is a simple $A_n$-module, this has to equal $V$.\\
Hence, we may assume from the start, without loss of generality and using the same notations as before, that $\lambda\circ\sigma^i=\lambda$ for all $i$. For $a\in S_i$, Lemma \ref{l.comp} (i) shows that $Tr_V(\theta(a))=\dim(V)\lambda(\theta(a))=0$ by hypothesis (i), since $\theta(a)$ is a commutator when viewed as an element of $\End(V)$. Thus, we again obtain that $\theta(a)$ acts nilpotently on $V$. By Lemma \ref{l.comp}, $V$ is a direct sum of $1$-dimesional $A_{n-1}$-modules, all of weight $\lambda$. As in the proof of Theorem \ref{t.genlie}, we choose $w$ to be an eigenvector for $x$, and since $ax=\lambda(a)x$ for $a\in A_{n-1}$ we obtain that $A_nw$ is an $A_n$-submodule of $V$. We conclude $V=A_nw$ is 1-dimensional by simplicity. The second assertion follows as before.  
\end{proof}

\noindent We note now that Lie's Theorem on the triangularizability of solvable Lie algebras can be obtained here as a corollary.

\begin{corollary}[Lie's Theorem]
Let $\KK$ be an algebraically closed field of characteristic 0, $V$ be a finite dimensional $\KK$-vector space, and $\mathfrak{g} \subset \End(V)$ a solvable Lie algebra. Then 
$V$ contains a $\mathfrak{g}$-flag.
\end{corollary}

\begin{proof}
As noted in Example \ref{e.solvlie}, $A = U(\mathfrak{g})$ is an Ore-solvable algebra. Furthermore, the relations $xy = yx + [x,y]$ show that at each step we can pick $\sigma=1$ and $\theta=[x,-]$, which satisfy the conditions of Theorem \ref{t.genlie}. Hence, every finite dimensional simple $A$-module is 1-dimensional and $V$ is $p$-finite (equivalently, it admits a $\mathfrak{g}$-flag).
\end{proof}

\noindent Note also that by Example \ref{e.solvgp}, the Ore datum of the group algebra $A = \KK G$ of a solvable group $G$ satisfies the hypotheses of Theorem \ref{t.genlie}. The proof of the above corollary can then be suitably adapted to recover, at least in the case of a field of characteristic zero, the related \emph{Lie-Kolchin theorem}: if $G$ is a connected, solvable linear algebraic group, then every finite dimensional representation of $G$ has a one dimensional $G$-invariant subspace. We leave the details to the reader.

\subsection{Arbitrary Characteristic and Examples}

\begin{theorem}\label{t.3}
Suppose $\KK$ is an algebraically closed field of arbitrary characteristic. Let $A$ be an Ore-solvable algebra, satisfying the following conditions: 
\begin{enumerate}
\item[(i)] For each $i$, $A_{i-1}$ is generated by a set $S_i$ such that for each $a\in S_i$, either $\theta_i(a)$ is nilpotent or $\theta_i(a) \in [A_{i-1},A_{i-1}]$.
\item[(ii)] For each $i$, the orbit of any 1-dimensional character under the action of the semigroup $\langle\sigma_i\rangle\subset \End(A_{i-1})$ generated by $\sigma_i$ is either trivial or infinite. 
\end{enumerate}
Then the conclusions of Theorem \ref{t.genlie} hold: every finite dimensional simple module is 1-dimensional, and an $A$-module $V$ is p-finite (equivalently, locally finite) if and only if every $x_i$ acts locally finite on $V$.  
\end{theorem}

\begin{proof} 
Proceed again by induction, and for the inductive step, let $A=A_{n-1}, B=A_n, x=x_n, \sigma=\sigma_n, \theta=\theta_n$. Let $V$ be a simple finite dimensional $B$-module. As before, there exists an $A$-eigenvector $v\in V$; using the simplicity of $V$ as a $B$-module, we obtain again a basis $\mathcal{B}=\{v,xv,\dots,x^{k-1}v\}$ of $V$. We note first that if $\theta_i(a)\in [A_{i-1},A_{i-1}]$ for some $a\in A_{i-1}$, then $\theta_i(a)$ is nilpotent. Indeed, using the notation of Lemma \ref{l.comp}, note that with respect to the basis $\mathcal{B}$, every element $a\in A$ is upper triangular. Thus, for $a,b\in A$, we have that $[a,b]$ is strictly upper triangular, and hence nilpotent. 
By the second part of Lemma \ref{l.comp} together with the fact that $V$ is a simple $B$-module, we have $V=\bigoplus\limits_{i=0}^{k-1}\KK x^iv$ as an $A_{n-1}$ module. In particular, this argument shows that if $w\in V$ is any eigenvector of $A$ of weight $\mu$, then $xw$ is also an eigenvector of $A$ with weight $\mu\circ\sigma$, so the simple $A$-module of weight $\mu\circ \sigma$ is also a submodule of $V$. Now, if the orbit of $\sigma$ on $\lambda$ is infinite (that is, if $\lambda\circ\sigma^i\neq \lambda\circ \sigma^j$ for $i\neq j$) it would follow that $V$ contains infinitely many types of non-isomorphic simple modules, which would necessarily form a direct sum inside $V$. This is impossible since $\dim(V)<\infty$. Therefore, hypothesis (b) shows that $\lambda\circ \sigma^i=\lambda$ for all $i$. We can now proceed as before in the proof of Theorem \ref{t.genlie} to pick any $x$-eigenvector $w$ and obtain $V=Bw$. 
The last part also follows as before. 
\end{proof}

\begin{example}[Quantum affine spaces at non-roots of 1]
Consider a quantum affine algebra $A=\KK\langle x_1,\dots,x_n\rangle/\langle x_ix_j-\omega_{ij}x_jx_i\rangle$, where each $\omega_{ij}$ is either equal to $1$ or otherwise is not a root of 1. This is the coordinate algebra of a quantum affine $n$-space. Consider the intermediate subalgebras $A_i$ generated by $x_1,\dots,x_i$. Obviously, this presents $A$ as an iterated Ore extension, and so also as an Ore-solvable algebra. The Ore datum is given by $\sigma_j(x_i)=\omega_{ij}x_i$ for $i<j$ and $\theta_i=0$. If $\lambda:A\rightarrow \KK$ is a 1-dimensional representation of $A_{j-1}$, with $\lambda(x_i)=\alpha_i$, then $\lambda\circ\sigma_j(x_i)=\omega_{ij}\alpha_i$ for $i<j$. Thus, $\lambda\circ\sigma^t=\lambda$ for some fixed $t$ only if $\omega_{ij}^t\alpha_i=\alpha_i$ for all $i<j$. If this is the case, then either $\omega_{ij}=1$, or otherwise, $\alpha_i=0$. But in this case we also get $\lambda\circ\sigma=\lambda$, and thus the orbit of $\langle\sigma\rangle$ on $\lambda$ is trivial. Therefore, hypothesis (ii) of Theorem \ref{t.3} is satisfied, and (i) holds trivially. The theorem then shows that every finite dimensional simple module over this quantum space is 1-dimensional, and a module is locally p-finite (i.e. triangularizable) if and only if each of the $x_i$'s act locally finite. This shows that ``generically" (in terms of their commuting constants), a finite collection of skew commuting triangularizable operators on an arbitrary vector space are simultaneously triangularizable (as long as their skew relations are by non-roots of unity). \\
In particular, this applies to a skew commuting polynomial ring  $\KK\langle x,y\rangle/\langle yx-qxy\rangle$ (quantum  plane) for $q$ such that either $q=1$ or $q$ is a non-root of unity. For a specific example, this recovers the known fact that if $A,B$ are $n\times n$ matrices over an algebraically closed field and $AB=qBA$ with $q$ a non-root of unity or $q=1$, then $A$ and $B$ are simultaneously triangularizable; it also shows that if $A,B$ are triangularizable operators on an arbitrary space such that $AB=qBA$, then they are simultaneously triangularizable.
\end{example}

\noindent The same example as above can be modified in various ways to allow for maps $\theta_i\neq 0$. We give just one such example to illustrate the idea.

\begin{example} 
The algebra $A=\KK\langle x,y,z\rangle/\langle yx-qxy, zx-rxz-[x,y], zy-syz-[x,y]\rangle$ with $q,r,s$ non-roots of unity, satisfies the hypotheses of Theorem \ref{t.3}, and so every finite dimensional simple module is 1-dimensional, and an $A$-module is locally p-finite if and only if each $x,y$ and $z$ acts locally finite on $V$. Equivalently, three triangularizable operators $x,y,z$ on a vector space satisfying these relations are simultaneously triangularizable.
\end{example}

\noindent  Another example is the coordinate algebra of quantum matrices.

\begin{example}[Quantum matrices] \label{e.qmatrices}
Let $\KK$ be an algebraically closed field of characteristic zero, and $q \in \KK^{\times}$ be a non-root of unity. For each natural number $n$, let $[n] := \{ 1,\ldots , n\}$. We let $M_q(m,n)$ denote the quotient of the free algebra on the generators $\{ X_{ij} \mid (i,j) \in [m]\times [n]\}$ by the following relations, which are defined for each pair $i<j$ and $k<l$: 
\begin{enumerate} 
\item $X_{il}X_{ik} = qX_{ik}X_{il}$,
\item $X_{jk}X_{ik} = qX_{ik}X_{jk}$, 
\item $[X_{jk},X_{il}] = 0$, 
\item $[X_{ik},X_{jl}] = (q^{-1}-q)X_{jk}X_{il}$.
\end{enumerate} 
Order the set $[m]\times [n]$ lexicographically, so that $(i,j) \le (k,l)$ if $i< k$ or $i = k$ and $j \le l$. For each $(i,j)$, let $A_{ij}$ denote the subalgebra of $M_q(m,n)$ generated by the elements $\{ X_{kl} \mid (k,l) \le (i,j)\}$. This collection of subalgebras is totally ordered under inclusion, with $A_{ij} \subset A_{kl}$ if and only if $(i,j) \le (k,l)$. The corresponding sequence of subalgebras shows that $M_q(m,n)$ is an iterated Ore extension, and so in particular it is (left and right) Ore-solvable. If for each $i< j$ and $k<l$ we further impose the condition that $X_{jk}X_{il}$ is nilpotent, then by Theorem \ref{t.3}, if $V$ is an $M_q(m,n)$-module, then $V$ is triangularizable if and only if each of the $X_{ij}$ are triangularizable when viewed as endomorphisms of $V$ (i.e. their action on $V$ is locally finite). 
\end{example}

\subsection{A converse statement}

\noindent We have seen how certain conditions on Ore-solvable algebras imply that finite dimensional simple modules are 1-dimensional; on the other hand, we note that every finite dimensional pointed algebra (an algebra such that every simple $A$-module is 1-dimensional) arises this way and so also from a Lie algebra.  We will need the following result of \cite{IS}: if $A$ is a pointed algebra, then $A$ has a subalgebra of codimension $1$. We note this can also be proved directly as follows: if $J$ is the Jacobson radical of $A$, then $A/J\cong \KK^n$ as algebras for some $n$, and $A/J\otimes A/J$ is semisimple. All $J^{k-1}/J^k$ are $A$-bimodules that are semisimple as bimodules (they are also $A/J$-bimodules) and whose simple components are 1-dimensional $A/J$-bimodules and $A$-bimodules. Thus, as an $A$-bimodule, $A$ is p-semiartinian, and hence we can find an ideal (equivalently, a sub-bimodule) $I$ of $A$ of codimension 2 ($A$ has a composition series as $A$-bimodules with 1-dimensional factors). Then, the subspace $B=\KK 1+I$ is obviously a subalgebra of $A$ of codimension 1. \\
Furthermore, we now note that such an extension of algebras $B\subseteq A$ with $B$ of codimension $1$ can be presented in a special way. Let $x\in A\setminus B$. Since $x\in Bx$, we then have $B\subsetneq Bx+B\subseteq A$ and so $Bx+B=A$, and similarly $xB+B=A$. Hence, for $b\in B$, write $b=\lambda 1+h$ (uniquely with $\lambda\in\KK, h\in I$) and note that $xb=x+xh$ and $bx=x+hx$ so $xb=bx+[x,h]$. Thus, setting $\sigma={\rm Id}_B:B\rightarrow B$ and $\theta:B\rightarrow B$ to be the map $\theta(b)=[x,h]$ where $b\in B$ is, as above, uniquely written as $b=\lambda\cdot 1+h$, $h\in I$, then $xb=\sigma(b)x+\theta(b)$. We can summarize this in the following.

\begin{theorem}
Let $A$ be a finite dimensional pointed algebra over an arbitrary field $\KK$. Then $A$ is Ore-solvable with special Ore relations, which are of Lie type. More specifically, there is a sequence $\KK=A_0\subset A_1\subset \dots\subset A_n=A$ of subalgebras and $x_i\in A_i \setminus A_{i-1}$, such that $\dim(A_i/A_{i-1})=1$, and such that $\sigma_i=\sigma_i'={\rm Id}_{A_{i-1}}$ and $\theta_i(a)=-\theta_i'(a)=[a,x_i]$ provide an Ore datum for this sequence.  
\end{theorem}
\begin{proof} 
One simply proceeds inductively using the previous observation.
\end{proof}

We also remark that we may interpret $A$ as a Lie algebra via the commutator bracket $[-,-]$; since $A$ is pointed, its regular representation is triangularizable, and we may thus embed $A$ into the triangular matrix algebra of some size. This is also a Lie algebra embedding, and thus $A$ is a solvable Lie algebra, which provides an alternate way to obtain that $A$ is Ore-solvable with special Ore relations.  

The following question may be interesting to investigate.

\begin{question}
In what other situations, beyond finite dimensional pointed algebras, can one say  that if an algebra is triangularizable, then it is an (left) Ore-solvable algebra? 
\end{question}




\section{Nil algebras and Strict Upper-Triangularizability} 

\subsection{Nilpotency Conditions and Strict Local Finiteness}

The results of the previous section can be regarded as non-strict triangularization results. We now give ``strict" triangularization results which generalize several known such results, and which are related to nilpotent conditions on the algebra. We again use a representation-theoretic interpretation of strict triangularization, which often has to do with the existence of a single isomorphism type of simple module. We keep the same notation as in the previous section for Ore-solvable algebras and Ore datum. The next theorem is a simultaneous strict triangularization theorem, and does not require any knowledge or assumption on the $\sigma_i$'s from the Ore datum. 


\begin{theorem}\label{t.na1}
Let $A$ be a left Ore-solvable algebra, $V$ an $A$-module and assume that the Ore datum of $A$ satisfies the following two conditions: 
\begin{enumerate}
\item Each algebra $A_i$ from the filtration of $A$ is generated by a set $S_i$, such that for each $a\in S_{i-1}$, $\theta_i(a)$ acts locally nilpotent on $V$; and
\item each $x_i$ acts locally nilpotent on $V$.  
\end{enumerate}
Then $V$ is locally p-finite homogeneous, and every simple subquotient of $V$ has character $\mu:A\rightarrow \KK$ which satisfies $\mu(x_i)=0$ for all $i$.   
\end{theorem} 

\begin{proof} 
Again we proceed by induction on $n$ to show that $V$ is locally p-finite. If $A = A_0 = \KK$ then the result is clear, so suppose $A = A_n$ for some $n>1$.  
By condition (2) and Theorem \ref{t.1st}, we have that $V$ is locally $A$-finite, and we need to show it is locally p-finite homogeneous. 
Let $W$ be a simple subquotient of $V$. It satisfies the same properties (1) and (2) of the action of $A$ as $V$ does.
By induction, $W$ contains a one-dimensional simple $A_{n-1}$-module, spanned by an eigenvector $v$. Let $\lambda : A_{n-1} \rightarrow \KK$ be the associated algebra character, and for simplicity, let us denote $x=x_n$, $\sigma=\sigma_n$ and $\theta=\theta_n$. Let $Y = \operatorname{span}_{\KK}\{ x^iv \mid i \geq 0 \}$ be the $\KK[x]=\KK [x_n]$-module generated by $v$, which is finite dimensional. By Lemma \ref{l.comp}, it is also an $A_n$-submodule, and so $Y=W$ by the simplicity of $W$. Applying  this Lemma again, we observe that $W=\bigoplus\limits_i \KK x^iv$, a direct sum of 1-dimensional $A_{n-1}$modules; since $x$ acts nilpotently on $W$, we must have $x^{n+1}v=0$ and $x^{n}v\neq 0$ for some $n$. But then $w=x^nv$ is both an $A_{n-1}$ and $x$-egienvector, and so $\KK x$ is a $B$-submodule of $W$. Thus $W=\KK v$ is 1-dimensional. Furthermore, we note that the character $\mu:W\rightarrow \KK$ of $W$ must have $\mu(x_i)=0$, since $x_i$ acts nilpotently on any finite dimensional subquotient of $V$ (and so on $W$). Since $A_n$ is generated by the $x_i$'s as an algebra, this completely determines $\mu$ and $W$, and proves that $V$ is homogeneous too. 
\end{proof}


\noindent We note that the above can be interpreted as a strict simultaneous triangularization result. An operator $x\in \End(V)$ is strictly upper triangularizable if and only if it is locally nilpotent, and we thus restate the above theorem in this form:

\begin{corollary}
Let $x_1,\dots,x_n\in \End(V)$ be linear endomorphisms of a vector space. Let $A_i$ be the algebra generated by $x_1,\dots,x_i$ and assume $A=A_n$ is Ore-solvable (with $x_ia=\sigma_i(a)x_i+\theta_i(a)$, for $a\in A_{i-1}$, $\sigma_i,\theta_i:A_{i-1}\rightarrow A_{i-1}$) satisfying condition (1) of the previous theorem. If $x_1,\dots,x_n$ are strictly triangularizable then they are simultaneously strictly triangularizable. 
\end{corollary}
\begin{proof} 
 As noted before, $A$ is triangularizable in $\End(V)$ if and only if $V$ is p-semiartinian. Equivalently, $V$ is locally p-finite (since $A$ is Noetherian).  
\end{proof}

\noindent In particular, the above also recovers the case of commuting or even skew commuting operators on an arbitrary vector space. 

\begin{corollary}
Let $x_1,\dots,x_n\in \End(V)$ be skew-commuting operators on some vector space $V$, so $x_ix_j=\lambda_{ij}x_jx_i$. If $x_i$ are strictly upper triangularizable (equivalently, they act locally nilpotent), then they are simultaneously strictly upper triangularizable.
\end{corollary}

\noindent We now note how the above theorem  generalizes Engel's theorem on finite dimensional Lie algebras. Recall that the \emph{lower central series} of a Lie algebra $(\mathfrak{g}, [\cdot , \cdot ])$ is defined recursively via $\mathfrak{g}_1 = [\mathfrak{g},\mathfrak{g}]$, and $\mathfrak{g}_{n+1} = [\mathfrak{g}_n,\mathfrak{g}]$ for $n \geq 1$. Recall that $\mathfrak{g}$ is \emph{nilpotent} if $\mathfrak{g}_n = 0$ for some finite $n$ and $\mathfrak{g}$ is \emph{nil} if each $x \in \mathfrak{g}$ is \emph{ad-nilpotent}, i.e. the linear endomorphism $[x, -] : \mathfrak{g} \rightarrow \mathfrak{g}$ is nilpotent. This means that $U(\mathfrak{g})$ acts on the vector space $V=\mathfrak{g}$ such that each $x\in\mathfrak{g}\subset U(\mathfrak{g})$ acts as a nilpotent operator on $V$. 

\begin{remark}
Let $\mathfrak{g}$ be a finite dimensional Lie algebra, and let $V$ be a finite dimensional $\mathfrak{g}$-module (equivalently, $U(\mathfrak{g})$-module). Assume that each $x\in \mathfrak{g}$ is nilpotent on $V$; then Engel's theorem states that the elements of $\mathfrak{g}$ are simultaneously strictly upper triangularizable in $\End(V)$. We observe now that $U(\mathfrak{g})$ is exactly the type of Ore-solvable algebra described in Theorem \ref{t.na1}, and hence Engel's theorem can be recovered from that result. (more precisely, the ideas of the proof of Engel's theorem  show that $U(\mathfrak{g})$ is such an algebra).
Indeed, by the (usual) proof of Engel's theorem (see \cite{H}),
 reduce first to the case when $\mathfrak{g}\subset {\rm gl}(V)$ and then observe that the hypothesis implies that the ad-action of $\mathfrak{g}$ on $\mathfrak{g}\subset {\rm gl}(V)$ is nilpotent. Now we can proceed inductively as in that proof to show that $U(\mathfrak{g})$ satisfies the hypotheses of Theorem \ref{t.na1}: consider a maximal subalgebra $0\neq \mathfrak{h}$ of $\mathfrak{g}$ (which exists), and apply the inductive hypothesis for $U(\mathfrak{h})$ and the $U(\mathfrak{h})$-module $\mathfrak{g}/\mathfrak{h}$; by Theorem \ref{t.na1}, this module is p-finite homogeneous, so we can find a 0-eigenvector $\overline{x}$ for $\mathfrak{h}$. This means $x\in \mathfrak{g}$ is such that $[\mathfrak{h},x]\subset \mathfrak{h}$ and maximality also shows that $\mathfrak{h}\oplus \KK x=\mathfrak{g}$. The relation $xa=ax+[x,a]$ for $a\in \mathfrak{h}$ with $[x,a]\in\mathfrak{h}$ shows that $U(\mathfrak{h})\subset U(\mathfrak{g})$ is an Ore extension with $\sigma={\rm id}$ and $\theta(a)=[x,a]$ such that $\theta(a)$ is nilpotent on $V$ for all $a\in \mathfrak{h}$. Thus the conditions of Theorem \ref{t.na1} are satisfied.
\end{remark}

\noindent  
Motivated by the above theorem, one might ask what can be said about algebras defined by relations such that every module satisfies the conditions of Theorem \ref{t.na1} (i.e. nilpotency relations). In that case, as we note, the algebra is finite dimensional; this is why Theorem \ref{t.na1} is formulated in terms of a module with locally finite actions, instead of in terms of given algebra relations. 


\begin{proposition}\label{p.OreNilAlg}
Let $A$ be 
a left Ore-solvable algebra whose Ore datum $(x_i,\sigma_i,\theta_i)_i$ satisfies the following conditions: 
\begin{enumerate}
\item[(i)] There is a set $S_i$ of generators of $A_i$ such that for each $a\in S_{i-1}$, $\theta_i(a)$  is nilpotent.
\item[(ii)] $x_i$ is nilpotent for each $i$. 
\end{enumerate}
Then $A$ is a finite dimensional local algebra.
\end{proposition}
\begin{proof}
Note that by Theorem \ref{t.na1} every $A$-module is locally p-finite, in particular semiartinian. Hence $A$ is semiartinian. Since $A$ is also Noetherian, 
it is therefore artinian. But $A$ is also locally p-finite and homogeneous, and its unique simple module (up to isomorphism) is 1-dimensional. Thus $A$ is local and finite dimensional. 
\end{proof}

\noindent We now proceed to observe that any local pointed (connected) finite dimensional algebra can in fact be presented as an algebra as in the previous proposition.

\begin{proposition}\label{FinDimLCA}
Let $A$ be a finite dimensional local pointed algebra. Then $A$ is an Ore-solvable algebra whose Ore datum $(x_i,\sigma_i,\theta_i)_i$ satisfies the conditions of the previous proposition; furthermore, the maps $\sigma_i$ may be chosen as the identity. 
\end{proposition}
\begin{proof}
The proof proceeds as before for finite dimensional algebras. Let $J$ be the Jacobson radical of $A$. As before, we may find a maximal two sided ideal $I$ of $J$ which must have codimension $1$ (consider $J$ as $A$-bimodule). Let $x\in J\setminus I$ and let $B=\KK 1\oplus I$; then for $a\in I$, $xa,ax\in I$, and so $[x,a]\in I$. Hence, $A$ is generated by $B$ and $x$ with Ore-type relation $xa=ax+[x,a]$ such that $[x,a]\in B$. Furthermore, it is straightforward to note that $B$ is also local pointed, since the ideal $I$ of $B$ has codimension $1$ and consists of nilpotent elements, so $I$ is the Jacobson radical of $B$ (one can also apply, for example, results of \cite{IS}). Inductively, one may thus show that $A$ is Ore solvable of the special type stated. 
\end{proof}

\begin{remark}
The above two propositions show that if a finite collection of operators $x_i$ of an arbitrary vector space $V$ satisfy the conditions of Proposition \ref{p.OreNilAlg} then they generate a finite dimensional subalgebra $A$ of $\End(V)$. Furthermore, the $x_i$ can be ``exchanged" for $y_j \in A$ relative to which the $\sigma_i$'s become identity maps and the $\theta_i$'s are given by $\theta_i(a)=[x_i,a]$. In addition, observe that the Lie subalgebra $L$ of $A$ generated by the $x_i's$ is nilpotent: since $J$ is nilpotent, each element in $J$ acts nilpotently on $A$, and so $L$ is a nilpotent Lie algebra by Engel's Theorem. While the finite dimensional case of Theorem \ref{t.na1} essentially reduces to Engel's theorem, one may still have situations in the infinite-dimensional case where the $x_i$'s and $\theta_i(a)$'s are locally nilpotent without being nilpotent.
\end{remark}

\begin{example} 
Consider the quantum matrix algebras $M_q(m,n)$ as in Example \ref{e.qmatrices}. The considerations there now show that if $V$ is a $M_q(m,n)$-module such that for each $i< j$ and $k<l$, $X_{jk}X_{il}$ is locally nilpotent, and for all $i,j$, $X_{ij}$ is locally nilpotent, then $V$ is locally p-finite. Equivalently, the simultaneous strict triangularization theorem for quantum matrices says that if $X_{ij}$ are strictly triangularizable, and if $X_{jk}X_{il}$ are also strictly triangularizable, then $X_{ij}$ are simultaneously strictly triangularizable. 
\end{example}



\section{Strict Triangularization and Locally Finite Modules}

\noindent In this section, we investigate other situations of infinite dimensional strict triangularization that reduce to a locally finite type of triangularization, and connections to other nil/nilpotency conditions. In module theoretic terms, these are algebras where semiartinian modules are locally finite. 
In what follows, for a non-unital algebra $A$ (i.e. an algebra that {\it does not} have a unit), we will let $A_u=\KK\oplus A$ denote the \emph{unitalization} of $A$. If $A$ is unital, then we set $A_u = A$.
We recall that a not necessarily unital algebra $H$ is called \emph{nil} if every element in $A$ is nilpotent. We say that a unital algebra $A$ is nil if $A$ is the unitalization of a nil algebra $H$; that is, $A=\KK\oplus H$ where $H$ is a nil-ideal of $A$. The definitions before on solvable or Ore-solvable algebras apply 
for not necessarily unital algebras; that is,  
a not necessarily unital algebra $H$ will be called left Ore solvable if there is a sequence of intermediate algebras $H_i$ and elements $x_i\in H_i$, such that $H_{i-1}x_i\subseteq x_iH_{i-1}+H_{i-1}$ and $H_i$ is generated by $H_{i-1}$ and $x_i$. In this case, under this commuting assumption, the last condition can be written as $H_i=\sum\limits_{n\geq 1}x_i^nH_{i-1}+H_{i-1}+x_i\KK[x_i]$. The right Ore-solvable algebras are defined similarly. 
We also note that given a non-unital solvable (resp. Ore-solvable) algebra $H$, its unitalization $A=\KK\oplus H$ will be solvable (resp. Ore-solvable) by defining the successive extensions $A_i$ as $A_i=\KK\oplus H_i$. Then $A$ is a semitrivial extension of $\KK$ and $H$. Also, if $A$ is a connected algebra, then $A=\KK\oplus M$ where $M$ is its unique cofinite maximal ideal, and $A=M_u$. 

\noindent As noted, triangularization is closely related to locally finite modules, and strict triangularization, to locally finite homogeneous simple modules, and consequently to algebras that are local. This motivates the following definition, which will be relevant for our results.

\begin{definition}
Let $A$ be a $\KK$-algebra. 
\begin{enumerate}
\item We will say that $A$ is connected if it has a unique left (equivalently, right) maximal ideal $M$ of finite codimension, and moreover, $\dim(A/M)=1$. Equivalently, up to isomorphism, $A$ has a unique finite dimensional irreducible representation which is moreover 1-dimensional.
\item We say that $A$ is connected local if it is connected and local (has a unique maximal left/right ideal). 
\item We say that $A$ is strongly connected if it is connected and every $A$-module is locally finite (and hence, locally p-finite in this case). 
\end{enumerate}
\end{definition}

\noindent Note that a connected algebra $A$ is connected local if and only if every simple $A$-module is finite dimensional, or equivalently, up to isomorphism, it has a unique simple module which is furthermore 1-dimensional; in this case, its Jacobson radical $J$ has co-dimension 1. Of course, an algebra can be connected in this sense without being local - take, for example, $\CC\times \mathcal \CC\langle x,y\rangle/\langle xy-yx -1\rangle$.  Note also that if $A$ is strongly connected then it is connected local, but the converse is not be true (consider for example the algebra of formal power series $\KK[[X]]$ over $\KK$).

\begin{proposition} \label{p.stcnil}
Let $A$ strongly connected algebra. Then $A$ is nil. 
\end{proposition} 

\begin{proof} 
By hypothesis, $A$ is a locally finite module over itself. Let $x \in J(A)$ be an element of the Jacobson radical. Then $Ax$ is finite-dimensional,  so we can find a composition series $0 = M_0 \le M_1\ldots \le M_k = Ax$ for $Ax$. Since $A$ is strongly connected, $J(A)\cdot (M_i/M_{i-1}) = 0$ for all $i \le k$. In particular, $J(A)^k\cdot (Ax) = 0$, and since $x \in J(A)$, $x^{k+1} =0$ as well. Since $A = \KK\cdot 1 \oplus J(A)$, it follows that $A$ is the unitalization of the nil algebra $J(A)$.
\end{proof}

 \noindent 
  When dealing with an extension of unital algebras $A\subseteq B$ and a subalgebra $H$ of $B$ with $AH+A=HA+A=B$ (or even $AH\subseteq HA+A=B$), then we view $B$ as an ``extension of $A$ by $H$". Of course, we obviously lack any ideal and quotient structure on $H$ as a $B$-module or quotient of $B$. The following lemma shows that under a suitable setup, an extension of a connected algebra by a connected algebra is again connected.

\begin{lemma}\label{l.connext} The following statements hold:
\begin{enumerate}
\item Let $A\subset B$ be an extension of non-unital algebras, $H$ a subalgebra of $B$ such that $AH\subseteq HA+A$ and $B$ is generated by $H$ and $A$. Suppose that $A_u$ and $H_u$ are connected algebras. Then $B_u$ is connected. 
\item Let $A\subseteq B$ an extension of unital algebras such that $B$ is a left extension of $A$ by $H$, i.e. $H$ is a subalgebra of $B$ such that $AH\subseteq HA+A$ and $B$ is generated by $H$ and $A$ (equivalently, $B=HA$). Suppose $A$ and $H$ are each strongly connected algebras. Then $B$ is strongly connected. 
\end{enumerate}
\end{lemma} 

\begin{proof} 
(1) Let $V$ be a finite-dimensional simple $B_u$-module. Then $V$ restricts to a finite-dimensional $A_u$-module, so that by hypothesis $V$ contains an $A_u$-eigenvector, call it $v$. Let $\lambda_A : A_u \rightarrow \KK$ be the corresponding character, so that $A_u = \KK \oplus \ker (\lambda_A) = \KK \oplus A$. Let $\lambda_H : H_u \rightarrow \KK$ be the unique character of $H$, so that $H_u = \KK \oplus \ker (\lambda_H) = \KK \oplus H$. Since $H_u\cdot v$ is a finite-dimensional $H_u$-module, it contains a simple $H_u$-submodule spanned by an eigenvector $xv$, where $x \in H_u$. For all $a \in A$, the condition $AH \subseteq HA+A$ implies that $ax = h'a'+ a''$ for suitable $a',a'' \in A$ and $h' \in H$. Therefore, $a(xv) = (h'a' + a'')v = 0$ (since $v$ is an $A_u$-eigenvector, with $Av=0$), and so it follows that $\KK xv$ is $A_u$-invariant too. Since $B$ is generated by $H$ and $A$, $\KK xv$ is also $B_u$-invariant, and so $V = \KK xv$. Note that each element of $A$ and $H$ acts as zero on $V$. It follows that the only finite-dimensional simple is the projection map $B_u = \KK \oplus B \rightarrow \KK$, and so $B_u$ is connected. The proof of the second claim is similar, and uses the proof of Theorem \ref{t.1st} showing that a module is $B$-locally finite if and only if it is locally finite both as an $A$ and as an $H$-module.  
\end{proof} 



\noindent We note here an example showing that if in (1) of the above Lemma we start with $B$ an extension of $A$ by $H$, with both $A$ and $H$ connected local, it does not follow that $B$ is connected if we do not impose the stronger hypothesis that 
the radical of $B$ is an extension of the radical of $A$ by the radical of $H$.  

\begin{example} 
Let $B = M_2(\KK)$. Then $B$ is generated as an algebra (even a not-necessarily unital one) by the pair of matrices 
\begin{equation*}
 X = \left(\begin{array}{cc} 0 & 1 \\ 0 & 0 \end{array} \right), Y = \left( \begin{array}{cc} 0 & 0 \\ 1 & 0 \end{array} \right).
 \end{equation*} 
 \noindent Consider the subalgebras $A = \operatorname{span}_{\KK}\{ I_2, X\}$ and $H = \operatorname{span}_{\KK}\{ I_2, Y\}$. Then $A$ and $H$ are both connected local, and $B = AH + H = HA + A$. However, $B$ is not itself connected local. This shows that the statement (1) in Lemma \ref{l.connext} cannot be weakened to the condition ``$A_uH_u \subset H_uA_u + A_u$''.  
\end{example}

\noindent Lemma \ref{l.connext} implies that iterated extensions of connected algebras remain connected.

\begin{corollary}\label{c.conn}
Suppose $\KK=A_0\subset A_1\subset \dots \subset A_n$ are non-unital algebras, $B_i\subseteq A_{i+1}$ are subalgebras such that $A_iB_i\subseteq B_iA_i+A_i$ for $i<n$ and $A_i$ is generated by $A_{i-1}$ and $B_i$ (i.e. $A_i=B_iA_i+A_i+B_i$). If $(B_i)_u$ is connected for all $i$, then $A_u$ is connected.
\end{corollary}

\noindent To state our main strict triangularization results, recall from \cite{M2} that a subset $X\subseteq \End(V)$ is said to be \emph{topologically nilpotent} if and only if for each sequence $(x_n)_n$ of elements of $X$ and any vector $v\in V$, there is $n$ such that $x_nx_{n-1}\dots x_1\cdot v=0$. This is equivalent to saying that the sequence $(x_nx_{n-1}\dots x_1)_n$ converges to $0$ in the cofinite topology of $V$; this is the linear topology on $\End{(V)}$ for which a basis of neighborhoods of $0$ consists of the subspaces $W^\perp=\{f\in \End(V) | f(W)=0\}$ (see \cite{IMR,M2}). One of the main results of \cite{M2} shows that a set $X$ is simultaneously strictly triangularizable if and only if it is topologically nilpotent. We first give a representation-theoretic interpretation of this condition, which also recovers this result of \cite{M2}. We note the ideas are similar at their core, the difference being in the language which is representation-theoretical here.   

\begin{theorem}\label{t.sut0}
Let $V$ be a vector space, $X$ be a subset of $\End(V)$, and $A=\KK\langle X\rangle$ be the subalgebra generated by $X$ inside $\End(V)$. Then the following assertions are equivalent. 
\begin{enumerate}
\item $X$ is strictly triangularizable.
\item $V$ is a p-semiartinian homogeneous $A$-module with simple module $S$ such that every $x\in X$ acts as $0$ on $S$ ($xS=0$). 
\item $X$ is topologically nilpotent.  
\end{enumerate}
\end{theorem}
\begin{proof}
(1)$\Leftrightarrow$(2) is already noted in Proposition \ref{p.strictTR}. \\
(3)$\Rightarrow$(2) We first prove that $V$ is p-semiartinian. Suppose that for some proper $A$-submodule $W$ of $V$, there is no non-trivial $0$-eigenvector of $X$ in $V/W$; that is, for every $w\in V \setminus W$, there exists $a\in X$ such that $aw\notin W$. Let $w\in V\setminus W$. Then there is $a_1\in X$ such that $a_1w\notin W$; applying the assumption again on $a_1w\notin W$, we obtain $a_2\in X$ such that $a_2a_1w\notin W$. Recursively, we then construct a sequence $(a_n)_n\subset X$ such that $a_n\dots a_2a_1w\notin W$ for all $n$, and in particular, $a_n\dots a_1w\neq 0$. This contradicts the topological nilpotency.  \\
The homogeneous property follows by noting that if $S=\KK \overline{v}$ is a 1-dimensional simple subquotient of $V$ on which $X$ acts topologically nilpotent, then for each $x\in X$, there is $\lambda\in\KK$ such that $x\overline{v}=\lambda \overline{v}$; but then by topological nilpotency, there is some $n$ with $x^nv=0$, so $0=x^n\overline{v}=\lambda^n \overline{v}$ and thus $\lambda=0$. Hence, we get $xS=0$.\\
(2)$\Rightarrow$(3) Proceed by transfinite induction on the ordinal $\alpha$ such that $V$ admits a composition series of length $\alpha$. If the statement is true for all $\beta<\alpha$, let $(x_n)_n$ be a sequence of elements of $X$, and $v\in V=V_\alpha$. Then $x_1v\in V_\beta$ for $\beta<\alpha$: this is obvious if $\alpha$ is a limit ordinal, and if $\alpha=\gamma+1$ is a successor, then $x_1v\in V_\gamma$ by the property of $V_{\gamma+1}/V_\gamma$, so we can take $\beta=\gamma$. By the induction hypothesis applied for $(x_n)_{n\geq 2}$ and $x_1v\in V_\beta$, we get $x_nx_{n-1}\dots x_2 (x_1v)=0$.
\end{proof}

\noindent As noted also in the beginning of our investigation, it is often of interest to consider triangularization of finitely generated algebras. In this case, we will note that again strict triangularization reduces to the stronger locally finite type of triangularization, but here without the Noetherian hypothesis, which will further motivate the statement of our main result. 

\begin{lemma}\label{l.Koenig}
Let $X$ be a finite subset of $\End(V)$. If $X$ is topologically nilpotent (equivalently, strictly triangularizable), then $V$ is a locally finite $A=\KK\langle X\rangle$-module, where $\KK \langle X \rangle$ denotes the subalgebra of $\End (V)$ generated by $X$.
\end{lemma}
\begin{proof} 
Let $v \in V$ be given, and let $X = \{ x_1,\ldots , x_n\}$. We build a rooted tree $\mathcal{T}$, whose vertices will be labeled by the sets $(X^k)_{k\geq 0}$ ($X^0=\{1\}$), as follows: to begin, create a vertex (the root) which we label $1$. Next, for each $x_i \in X$, add a vertex labeled $x_i$ if $x_iv \neq 0$, and connect each such vertex to $1$ via an edge. Now, suppose that for all $k\geq 1$ and all $k$-tuples $\tau = (x_{i_k},\ldots , x_{i_1}) \in X^k$ we have defined the vertex $\tau$. To build the $X^{k+1}$ level of $\mathcal{T}$, we add new vertices as follows: for each $x_i \in X$, create a vertex labeled $(x_i,\tau ) := (x_i,x_{i_k},\ldots , x_{i_1}) \in X^{k+1}$ if and only if $(x_ix_{i_k}\cdots x_{i_1})v \neq 0$, and connect such a vertex to vertex $\tau$. This procedure produces a locally-finite connected graph, which, by the topological nilpotency of $X$, can contain no infinite ray. By K\"onig's Lemma (see for example \cite{T}), 
 it follows that this graph has finitely-many vertices. But then this shows that $Av$ has a finite spanning set, and so it is finite dimensional. Therefore, $V$ is locally finite.
\end{proof}


\begin{corollary}\label{c.connloc}
Let $A$ be a finitely-generated connected algebra. Then every f-semiartinian (equivalently, p-semiartinian) $A$-module is locally finite (hence locally p-finite). 
\end{corollary} 

\begin{proof} 
Since $A$ is connected, finite dimensional simple modules are 1-dimensional, so let $V$ be p-semiartinian, with $v$ a non-zero element of $V$. Write $A=\KK\oplus M$, with $M$ its maximal cofinite ideal, and write $\epsilon : A \rightarrow \KK$ for the projection afforded by this decomposition. After potentially taking the quotient of $A$ by the annihilator of $V$, we consider $A$ as a subalgebra of $\End(V)$. Note that $M$ acts as 0 on the simple, and so $M$ is topologically nilpotent by Theorem \ref{t.sut0}. Furthermore, if $\{ a_1,\ldots , a_n \}$ is a finite generating set for $A$, then $\{ a_1 - \epsilon (a_1), \ldots , a_n-\epsilon (a_n)\}$ is a finite generating set for $A$ contained in $M$. It follows that $M$ itself is finitely generated (as both a non-unital algebra and ideal of $A$), and so Lemma \ref{l.Koenig} applies.   
\end{proof}

\noindent Before giving our main result of this section, we note a converse to Proposition \ref{p.stcnil}.

\begin{lemma}\label{l.nconn}
Let $H$ be a nil algebra, and $A=H_u$. Then $A$ is connected local. If moreover $H$ is finitely generated (equivalently, $A$ is finitely generated), then $A$ is strongly connected.
\end{lemma}
\begin{proof}
Note that $H$ is an ideal of $A$, which is nil, so it is contained in the Jacobson radical, and thus $H=J(A)$ since $H$ has codimension $1$. Suppose in addition that $H$ is finitely generated. Since $A/J(A)=\KK$, every semiartinian module is f-semiartinian, and by Corollary \ref{c.connloc}, it has to be locally p-finite. Therefore, $A$ is strongly connected. 
\end{proof}


It may be useful to note that by the celebrated examples by Golod and Shafarevich \cite{GS}, there are finitely generated nil algebras which are not finite dimensional nor nilpotent. See also \cite{LS}.

\begin{definition}\label{d.RecNilAlg}
Let $\mathcal{A}$ be a class of non-unital algebras. We define left/right recursively-$\mathcal{A}$ algebras inductively as follows:
\begin{enumerate}
\item[(i)] $A$ is called 0-$\mathcal{A}$ if $A$ is in $\mathcal{A}$. We say that the set $c(A)=\{A\}$ is the set of the constituents of $A$. 
\item[(ii)] $A$ is left (respectively, right) $n$-$\mathcal{A}$ if there is a sequence of subalgebras $0=A_0\subset A_1\subset\dots \subset A_n$ and $B_i, \, 0\leq i\leq n$, such that $A_{i-1}B_i\subseteq B_iA_{i-1}+A_{i-1}$ (respectively, $B_iA_{i-1}\subseteq A_{i-1}B_i+A_{i-1}$), $A_i$ is generated by $B_i$ and $A_{i-1}$, and each $B_i$ is left (respectively, right) $k$-$\mathcal{A}$ for some $k=k(i)\leq n-1$ (depending on $i$). We say that the set $c(A)=\bigcup\limits_{i}c(B_i)$ is the set of constituents of $A$. 
\item[(iii)] We say that an algebra is left/right recursively-$\mathcal{A}$, or left/right R-$\mathcal{A}$, if $A$ is left/right $n$-$\mathcal{A}$ for some natural number $n$, and that it is R-$\mathcal{A}$ if it is both left and right R-$\mathcal{A}$. 
\item[(iv)] Finally, if $A$ is a unital algebra, we say that it is left/right/two-sided R-$\mathcal{A}$ if $A$ is the unitalization of a left/right/two-sided R-$\mathcal{A}$ algebra. 
\end{enumerate}
\end{definition}


We list here some instances of this definition. 
We will say that an algebra is recursively nil, or R-nil, if it is R-$\mathcal{A}$ where $\mathcal{A}$ is the class of nil algebras; we say that an algebra is recursively commutative, or R-commutative if it is R-$\mathcal{A}$ where $\mathcal{A}$ is the class of commutative algebras. We note that by the above definition, a solvable algebra is nothing but a 1-commutative algebra. If $\mathcal{A}$ is the class of cyclic algebras (by which we simply mean algebras generated by a single element), we say that an R-$\mathcal{A}$ algebra is an R-cyclic algebra. In this case, its constituents will be cyclic algebras $H_i=\langle x_i\rangle=x_i\KK[x_i]$ (or $H_i=\KK[x_i]$ if we are in the unital case), and we will also say that the set of constituent elements is the set of elements $(x_i)$. For another example, note also that a finitely generated commutative algebra is a recursively cyclic algebra (or R-cyclic algebra), and moreover, an Ore-solvable algebra is also an R-cyclic algebra (in fact, it is a 1-cyclic algebra). Also, for another example, Corollary \ref{c.conn} shows that the class of recursively connected algebras coincides to that of connected algebras, and that of recursively finitely generated algebras with that of finitely generated algebras. 

We now observe another very general situation where strict triangularization reduces to the more special locally finite type of triangularization.   

\begin{proposition}\label{t.sut}
Let $A$ be a left recursively nil algebra. Then $A_u$ is connected. 
Moreover, if $A$ is finitely generated, then every f-semiartinian (equivalently, block triangularizable) $A$-module $V$ is locally p-finite homogeneous.
\end{proposition} 
\begin{proof} 
Any nil algebra is connected by Lemma \ref{l.nconn}, and the connectedness of $A_u$ follows directly from 
 Corollary \ref{c.conn} (see the remark above). The second claim then follows from Corollary \ref{c.connloc}.
\end{proof}

We now note the following general simultaneous strict triangularization result. The class of algebras it applies to may be regarded as an counterpart of polycyclic-by-finite groups. Namely, we say an algebra $A$ is recursively cyclic-finite, or recursively cyclic-by-finite, (although there is no specific order between the terms cyclic and finite) if $A$ is recursively {$\{$cyclic or finite dimensional$\}$}. As noted before, this class of algebras includes many algebras, such as commutative, skew-commutative or Ore-solvable algebras, alongside finite dimensional algebras and various others built up from these. The following result can also be seen as an extension of some well known facts, such as that finitely many commuting nilpotent endomorphisms can be simultaneously strictly triangularized (see also \cite{M1}), or finite dimensional nilpotent subalgebras of $\End(V)$ can be simultanelously strictly triangularized, or Levitski's theorem on finite semigroups of nilpotent operators of a finite dimensional vector space (see also \cite{M2}). We recall to the reader that all these can be viewed as consequence of a general representation theoretic fact: if $A$ is a finite dimensional local algebra, then every $A$-module is locally p-finite homogeneous, and hence, strictly triangularizable. (This follows also from the more general Lemma \ref{l.nconn}, since a finite dimensional local algebra is nil in our terminology).

\begin{theorem}\label{t.strict}
Let $A$ be a recursively cyclic-finite algebra. Then an $A$-module $V$ is locally p-finite homogeneous (equivalently, in this case, strictly triangularizable) if and only if each of its constituents acts locally nilpotent on $V$. 
\end{theorem}
\begin{proof}
We only need to prove the if part. Also, it does no harm to identify $A$ with its image in $\End(V)$. In that case, each of its constituents is a nil algebra, $A$ is recursively nil, and therefore, connected by the previous Proposition. Since $V$ is a locally finite module over each of the constituents of $A$, applying Theorem \ref{t.1st} inductively (as $A$ is obtained by successive extensions), we obtain that $V$ is locally finite. Finally, the second part of the previous Proposition now shows that $V$ is locally p-finite homogeneous.  \\
(Note that since $A$ is finitely generated connected, by Corollary \ref{c.connloc}, we get that a module is triangularizable if and only if it is locally p-finite homogeneous). 
\end{proof}

Finally, also note another result which might be of interest from the perspective of the theory of nil algebras. In the next statement, finitely-generated-nil refers to the class of algebras that are nil and finitely generated.

\begin{corollary}
A recursively finitely-generated-nil algebra is nil. 
\end{corollary}
\begin{proof}
The proof goes along the same lines as that of Theorem \ref{t.strict} above. First, as there, such an algebra $A$ is connected. Let $V$ be an $A$-module. Each of the constituents of $A$ are finitely generated nil, and therefore strongly connected by Lemma \ref{l.nconn}. Thus, $V$ is locally finite (and in fact locally p-finite) as a module over each of these constituents, and again by Theorem \ref{t.1st}, we may inductively prove that $V$ is locally finite also as an $A$-module (and thus f-semiartinian). But $A$ is also finitely generated, and connected, and  Corollary \ref{c.connloc} shows that $V$ is locally p-finite. Hence, $A$ is strongly connected by definition. Finally, Proposition \ref{p.stcnil} implies that $A$ is nil. 
\end{proof}

\bigskip\bigskip\bigskip

\begin{center}
{\sc Acknowledgment}
M.C.I acknowledges the support of Simons Collaboration Grant 637866 for several visits to A.S. 

\end{center}


\vspace*{3mm} 
\begin{flushright}
\begin{minipage}{148mm}\sc\footnotesize

Jeremy Edison\\
Mount Mary University\\
Department of Mathematics, Fidelis Hall\\
Milwaukee, WI, USA\\
and\\
University of Iowa\\
Department of Mathematics, MacLean Hall \\
Iowa City, IA, USA\\%
{\it E--mail address}: {\tt
edisonj@mtmary.edu}\vspace*{3mm}

Miodrag Cristian Iovanov\\
University of Iowa\\
Department of Mathematics, MacLean Hall \\
Iowa City, IA, USA\\
and\\
Simion Stoilow Institute of the Romanian Academy,\\ 
010702 Bucharest, Romania\\
{\it E--mail address}: {\tt
yovanov@gmail.com; miodrag-iovanov@uiowa.edu}\vspace*{3mm}

Alexander Sistko\\ 
Manhattan College \\ 
Department of Mathematics, Research and Learning Center \\ 
Riverdale, NY, USA\\ 
and \\
Department of 
University of Iowa\\
Department of Mathematics, MacLean Hall \\
Iowa City, IA, USA\\%
{\it E--mail address}: {\tt
asistko01@manhattan.edu}\vspace*{3mm}

\end{minipage}
\end{flushright}
\end{document}